\documentclass[11pt]{amsart}
\usepackage{amsfonts,amssymb,amsthm}
\usepackage{amsmath,amscd}
\usepackage{pstricks}
\usepackage{pstricks,pst-node}
\usepackage{mathrsfs}
\usepackage[all]{xy}

\DeclareMathAlphabet{\mathpzc}{OT1}{pzc}{m}{it}

\renewcommand{\subsection}[1]{\vspace{.18in}
\par\noindent\addtocounter{subsection}{1}
\setcounter{equation}{0}{\bf\thesubsection.\hspace{5pt}#1}}

\theoremstyle{definition}
\newtheorem{Def}[subsection]{Definition}
\newtheorem{Rem}[subsection]{Remark}

\theoremstyle{plain}
\newtheorem{Prop}[subsection]{Proposition}
\newtheorem{Thm}[subsection]{Theorem}

\newtheorem{Lem}[subsection]{Lemma}
\newtheorem{Cor}[subsection]{Corollary}
\numberwithin{equation}{subsection}

\newcommand{\bfi}{{\mathbf{i}}}
\newcommand{\bfj}{{\mathbf{j}}}

\newcommand{\bfs}{{\mathbf{s}}}

\newcommand{\bfP}{{\mathbf{P}}}
\newcommand{\bfQ}{{\mathbf{Q}}}

\def\fS{{\frak S}}
\def\fT{{\frak T}}

\def\fkm{{\frak m}}

\newcommand{\ms}{\mathscr}

\def\sfz{{\mathsf z}}

\def\sH{{\mathcal H}}

\def\sL{{\mathcal L}}

\def\sP{{\mathcal P}}
\def\sQ{{\mathcal Q}}

\def\sS{{\mathcal S}}
\def\sT{{\mathcal T}}

\def\sX{{\mathcal X}}
\def\sY{{\mathcal Y}}

\newcommand{\mbn}{\mathbb N}

\newcommand{\mbc}{\mathbb C}
\newcommand{\mbz}{\mathbb Z}

\newcommand{\ttx}{\mathtt{x}}
\newcommand{\ttg}{\mathtt{g}}
\newcommand{\tts}{\mathtt{s}}
\newcommand{\ttt}{\mathtt{t}}

\newcommand{\ttu}{\mathtt{u}}

\newcommand{\ttU}{\mathtt{U}}
\newcommand{\ttV}{\mathtt{V}}
\newcommand{\ttT}{\mathtt{T}}
\newcommand{\ttS}{\mathtt{S}}
\newcommand{\ttf}{\mathtt{f}}
\newcommand{\ttk}{\mathtt{k}}
\newcommand{\tth}{\mathtt{h}}

\newcommand{\Hom}{\operatorname{Hom}}

\newcommand{\res}{\operatorname{res}}

\def\hMod{{\text-}{\mathsf{Mod}}}

\newcommand{\la}{{\lambda}}
\newcommand{\La}{\Lambda}

\newcommand{\dt}{\delta}
\newcommand{\Dt}{\Delta}

\newcommand{\Og}{\Omega}
\newcommand{\og}{\omega}

\newcommand{\up}{q}
\newcommand{\vi}{\varphi}
\newcommand{\ep}{\varepsilon}
\newcommand{\al}{\alpha}

\newcommand{\sg}{\sigma}

\def\th{\theta} 

\newcommand{\leb}{\left[}

\newcommand{\rib}{\right]}

\def\ggp#1#2{\left[\kern-3.2pt\left[{#1\atop #2}\right]\kern-3.2pt\right]}

\def\leq{\leqslant}\def\geq{\geqslant}
\def\le{\leqslant}

\newcommand{\trir}{\vartriangleright}

\newcommand{\ot}{\otimes}

\newcommand{\ti}{\widetilde}

\newcommand{\Lanr}{\Lambda(n,r)}

\newcommand{\lra}{\longrightarrow}
\newcommand{\ra}{\rightarrow}

\newcommand{\lm}{\longmapsto}

\newcommand{\vtg}{{\!\vartriangle}}

\newcommand{\DC}{{\frak D}_{\vtg,\mathbb C}}

\def\ttv{{q}}
\newcommand{\OgnC}{\Og_{n,\mathbb C}}
\newcommand{\OgC}{\Og_{\mathbb C}}
\newcommand{\SrC}{\sS(n,r)_\mbc}
\newcommand{\UglC}{{\rm U}_{\mathbb C}({\frak{gl}}_n)}
\newcommand{\afHrC}{\sH_{\vtg}(r)_\mathbb C}
\newcommand{\afUslC}{\text{\rm U}_{\mathbb C}(\widehat{\frak{sl}}_n)}
\newcommand{\afUglC}{\text{\rm U}_{\mathbb C}(\widehat{\frak{gl}}_n)}

\newcommand{\HrC}{\sH(r)_{\mathbb C}}

\newcommand{\afUnC}{U_{\vtg}(n)_\mbc}
\newcommand{\UnC}{U(n)_\mbc}

\newcommand{\afzrC}{{\zeta}_{\vtg,r}}
\newcommand{\zrC}{{\zeta}_{r}}

\newcommand{\afSrC}{{\mathcal S}_{\vtg}(n,r)_{\mathbb{C}}}

\newcommand{\afsl}{\widehat{\frak{sl}}_n}

\newcommand{\tri}{\triangle(n)}

\newcommand{\evSa}{\widetilde{\mathsf{ev}}_a}
\newcommand{\evHa}{\mathsf{ev}_a}
\newcommand{\evUa}{{{\mathsf{Ev}}}_a}

\def\res{{\text{res}}}
\newcommand{\Lcp}{\bar L}

\begin{document}
\title{Small Representations for Affine $q$-Schur Algebras}
\author{Jie Du}
\address{School of Mathematics, University of New South Wales,
Sydney 2052, Australia.} \email{j.du@unsw.edu.au}

\author{Qiang Fu}
\address{Department of Mathematics, Tongji University, Shanghai, 200092, China.}
\email{q.fu@hotmail.com}

\date{\today}

\thanks{Supported by the UNSW 2011 Goldstar Award and the National Natural Science Foundation
of China, the Program NCET, and the Fundamental Research Funds for the Central Universities}

\begin{abstract} When the parameter $q\in\mbc^*$ is not a root of unity, simple modules of affine $q$-Schur algebras have been classified in terms of Frenkel--Mukhin's dominant Drinfeld polynomials (\cite[4.6.8]{DDF}).
We compute these Drinfeld polynomials associated with the simple modules of an affine $q$-Schur algebra which come from the simple modules of the corresponding $q$-Schur algebra via the evaluation maps.
\end{abstract}

 \sloppy \maketitle
\section{Introduction}
Small representations of quantum affine $\mathfrak{sl}_n$ are the representations which are
irreducible when regarded as representations of (non-affine) quantum $\mathfrak{sl}_n$. In other words,
 by the evaluation maps \cite{Ji} from $\afUslC$ to $\UglC$, these representations are obtained from irreducible representations of quantum $\mathfrak{gl}_n$. Small representations have been identified by Chari--Pressley \cite{CP94} in terms of Drinfeld polynomials whose roots are described explicitly but fairly complicatedly.

  When the parameter is not a root of unity, simple representations of affine $q$-Schur algebras have also been classified \cite[Ch.~4]{DDF} in terms of finite dimensional simple polynomial representation for the quantum loop algebras of $\mathfrak{gl}_n$. These representations are labeled by dominant Drinfeld polynomials in the sense of \cite{FM}. Using the evaluation maps from affine $q$-Schur algebras to $q$-Schur algebras, every irreducible representation of a $q$-Schur algebra becomes an irreducible representations of the corresponding affine $q$-Schur algebra. Motivated by the work of \cite{CP94}, we will identify these small representations of affine $q$-Schur algebras in this paper by working out precisely their associated dominant Drinfeld polynomials.

Our method uses directly evaluation maps from affine $q$-Schur algebras to $q$-Schur algebras. By a compatibility relation with evaluation maps for quantum $\afsl$ and $\mathfrak{gl}_n$ (Proposition~\ref{commuting equation}), we will reproduce Chari--Pressley's result \cite[3.5]{CP94}
with simplified formulas for the roots of Drinfeld polynomials associated with small representations of $\afUslC$. In this way, the dominant Drinfeld polynomials for small representations of affine $q$-Schur algebras can be easily described by their roots in segments; see Corollary~\ref{DPforSR}.

 In a forthcoming paper, we will look at a more general question. Since every simple representations of an affine $q$-Schur algebra can be obtained by a generalized evaluation map from a simple representation of a certain cyclotomic $q$-Schur algebra introduced by Lin--Rui \cite{LR}, it would be interesting to classify those which are inflated from {\it semisimple cyclotomic} $q$-Schur algebras. By identifying them, we would be able to construct a completely reducible full subcategory of finite dimensional modules of an affine $q$-Schur algebra.

The paper is organized as follows. The first three sections are preliminary. Starting with the definitions of the double Ringel-Hall algebra $\DC(n)$ and the quantum loop algebra $\afUglC$ of $\mathfrak{gl}_n$ and an isomorphism between them in \S2, we discuss in \S3 polynomial representations of $\afUglC$ and the tensor space representations of $\DC(n)$ and present a classification of simple modules for the affine $q$-Schur algebras.
In \S4, we look at the classification of simple modules of the affine $q$-Schur algebras arising from representations of affine Hecke algebra.
Evaluation maps from affine $q$-Schur algebras to $q$-Schur algebras are defined in \S5 and we also prove a certain compatibility identity associated with the evaluation maps for quantum groups. In \S6, we reproduce
naturally a result of Chari--Pressley by simplifying the formulas for the roots of Drinfeld polynomials associated with small representations of quantum $\afsl$.
Dominant Drinfeld polynomials associated with small representations of affine $q$-Schur algebras are computed in \S7. We end the paper with an application to representations of affine Hecke algebras.

{\it Throughout the paper, $\ttv\in\mathbb C^*:=\mbc\backslash\{0\}$ denote a complex number which is not a root of unity.} For $m,n\in\mbz^+:=\mbn\backslash\{0\}$, $m\leq n$, let
$$[m,n]=\{m,m+1,\ldots,n\},\quad  [n]_q=\frac{q^n-q^{-n}}{q-q^{-1}},\;\text{ and }\;\biggl[{n\atop m}\biggr]=\frac{[n]_q[n-1]_q\cdots [n-m+1]_q}{[m]_q[m-1]_q\cdots [1]_q}.$$
All algebras are over $\mbc$.

\section{Quantum loop algebras: a double Ringel--Hall algebra interpretation}

We define two algebras by their generators and relations and give an explicit isomorphism between them.
The first algebra is constructed as a Drinfeld double of two extended Ringel--Hall algebras.

Let $(c_{i,j})$ be the Cartan matrix of affine type $A$ and let
$I=\mbz/n\mbz=\{1,2,\ldots,n\}$.

\begin{Def}[{\cite[2.2.3]{DDF}}]\label{presentation dHallAlg} The double Ringel--Hall algebra $\DC(n)$ of
the cyclic quiver $\tri$ is the $\mbc$-algebra generated by
$E_i,\ F_i,\  K_i,\ K_i^{-1},\ \sfz^+_s,\ \sfz^-_s,$ for $i\in I,\
s\in\mbz^+$, and relations:
\begin{itemize}
\item[(QGL1)] $K_{i}K_{j}=K_{j}K_{i},\ K_{i}K_{i}^{-1}=1$;

\item[(QGL2)] $K_{i}E_j=\up^{\dt_{i,j}-\dt_{i,j+1}}E_jK_{i}$,
$K_{i}F_j=\up^{-\dt_{i, j}+\dt_{ i,j+1}} F_jK_i$;

\item[(QGL3)] $E_iF_j-F_jE_i=\delta_{i,j}\frac
{\ti K_{i}-{\ti K_{i}}^{-1}}{\up-\up^{-1}}$, where $\ti K_i=
K_iK_{i+1}^{-1}$;

\item[(QGL4)]
$\displaystyle\sum_{a+b=1-c_{i,j}}(-1)^a\leb{1-c_{i,j}\atop a}\rib_q
E_i^{a}E_jE_i^{b}=0$ for $i\not=j$;

\item[(QGL5)]
$\displaystyle\sum_{a+b=1-c_{i,j}}(-1)^a\leb{1-c_{i,j}\atop a}\rib_q
F_i^{a}F_jF_i^{b}=0$ for $i\not=j$;

\item[(QGL6)] $\sfz^+_s\sfz^+_t=\sfz^+_t\sfz^+_s$,$\sfz^-_s\sfz^-_t=\sfz^-_t\sfz^-_s$,  $\sfz^+_s\sfz^-_t=\sfz^-_t\sfz^+_s$;
 \item[(QGL7)] $K_i\sfz^+_s=\sfz^+_s K_i$, $K_i\sfz^-_s=\sfz^-_s K_i$;

 \item[(QGL8)] $E_i\sfz^+_s=\sfz^+_s E_i$, $E_i\sfz^-_s=\sfz^-_s E_i,$ $F_i\sfz^-_s=\sfz^-_s F_i$, and  $\sfz^+_s F_i=F_i\sfz^+_s$,
\end{itemize}
where $i,j\in I$ and $s,t\in \mbz^+$.
It is a Hopf algebra with
comultiplication $\Dt$, counit $\ep$, and antipode $\sg$ defined
by
\begin{eqnarray*}\label{Hopf}
&\Delta(E_i)=E_i\otimes\ti K_i+1\otimes
E_i,\quad\Delta(F_i)=F_i\otimes
1+\ti K_i^{-1}\otimes F_i,&\\
&\Delta(K^{\pm 1}_i)=K^{\pm 1}_i\otimes K^{\pm 1}_i,\quad
\Delta(\sfz_s^\pm)=\sfz_s^\pm\otimes1+1\otimes
\sfz_s^\pm;&\\
&\varepsilon(E_i)=\varepsilon(F_i)=0=\varepsilon(\sfz_s^\pm),
\quad \varepsilon(K_i)=1;&\\
&\sg(E_i)=-E_i\ti K_i^{-1},\quad \sg(F_i)=-\ti K_iF_i,\quad
\sg(K^{\pm 1}_i)=K^{\mp 1}_i,&\\
&\text{and}\;\;\sg(\sfz_s^\pm)=-\sfz_s^\pm,&
\end{eqnarray*}
 where $i\in I$ and $s\in \mbz^+$.


Let $\afUslC$ be the subalgebra generated by $E_i,\ F_i,\  \ti K_i,\ti K_i^{-1}$ for $i\in [1,n]$. This is the quantum affine $\frak {sl}_n$.
Let $\afUnC$ (resp., $U(n)_\mbc$) be the subalgebra generated by $E_i,\ F_i,\  K_j,\ K_j^{-1}$ for $i,j\in [1,n]$ (resp., $i\in[1,n),j\in[1,n]$). This is the (extended) quantum
affine $\mathfrak{sl}_n$ (resp. quantum $\mathfrak{gl}_n$).
\end{Def}

The second algebra follows Drinfeld \cite{Dr88}.

\begin{Def}\label{QLA}
 The {\it quantum loop algebra} $\afUglC$  (or {\it
quantum affine $\mathfrak {gl}_n$}) is the $\mbc$-algebra generated by $\ttx^\pm_{i,s}$
($1\leq i<n$, $s\in\mbz$), $\ttk_i^{\pm1}$ and $\ttg_{i,t}$ ($1\leq
i\leq n$, $t\in\mbz\backslash\{0\}$) with the following relations:
\begin{itemize}
 \item[(QLA1)] $\ttk_i\ttk_i^{-1}=1=\ttk_i^{-1}\ttk_i,\,\;[\ttk_i,\ttk_j]=0$;
 \item[(QLA2)]
 $\ttk_i\ttx^\pm_{j,s}=\up^{\pm(\dt_{i,j}-\dt_{i,j+1})}\ttx^\pm_{j,s}\ttk_i,\;
               [\ttk_i,\ttg_{j,s}]=0$;
 \item[(QLA3)] $[\ttg_{i,s},\ttx^\pm_{j,t}]
               =\begin{cases}0\;\;&\text{if $i\not=j,\,j+1$},\\
                  \pm \up^{-is}\frac{[s]_q}{s}\ttx^\pm_{i,s+t},\;\;\;&\text{if $i=j$},\\
                  \mp \up^{(1-i)s}\frac{[s]_q}{s}\ttx_{i-1,s+t}^\pm,\;\;\;&\text{if $i=j+1$;}
                \end{cases}$
 \item[(QLA4)] $[\ttg_{i,s},\ttg_{j,t}]=0$;
 \item[(QLA5)] $[\ttx_{i,s}^+,\ttx_{j,t}^-]=\dt_{i,j}\frac{\phi^+_{i,s+t}-\phi^-_{i,s+t}}{\up-\up^{-1}}$;
 \item[(QLA6)] $\ttx^\pm_{i,s}\ttx^\pm_{j,t}=\ttx^\pm_{j,t}\ttx^\pm_{i,s}$, for $|i-j|>1$, and
 $[\ttx_{i,s+1}^\pm,\ttx^\pm_{j,t}]_{\up^{\pm c_{ij}}}
               =-[\ttx_{j,t+1}^\pm,\ttx^\pm_{i,s}]_{\up^{\pm c_{ij}}}$;
 \item[(QLA7)] $[\ttx_{i,s}^\pm,[\ttx^\pm_{j,t},\ttx^\pm_{i,p}]_\up]_\up
               =-[\ttx_{i,p}^\pm,[\ttx^\pm_{j,t},\ttx^\pm_{i,s}]_\up]_\up\;$ for
               $|i-j|=1$, \index{defining relations!({\rm QLA1})--({\rm QLA7})}
\end{itemize}
 where $[x,y]_a=xy-ayx$ and $\phi_{i,s}^\pm$ are defined by the
 generating functions in indeterminate $u$:
$$\Phi_i^\pm(u):={\ti\ttk}_i^{\pm 1}
\exp(\pm(\up-\up^{-1})\sum_{m\geq 1}\tth_{i,\pm m}u^{\pm m})=\sum_{s\geq
0}\phi_{i,\pm s}^\pm u^{\pm s}$$ with $\ti\ttk_i=\ttk_i/\ttk_{i+1}$
($\ttk_{n+1}=\ttk_1$) and
$\tth_{i,m}=\up^{(i-1)m}\ttg_{i,m}-\up^{(i+1)m}\ttg_{i+1,m}\,(1\leq
i<n).$
\end{Def}

Beck \cite{Be} proved that the subalgebra, called the quantum loop algebra of $\mathfrak{sl}_n$, generated by all $\ttx^\pm_{i,s}$, $\ti\ttk_i^{\pm1}$ and $\tth_{i,t}$ is isomorphic to $\afUslC$.

Set, for each $s\in\mbz^+$ and each $i\in[1,n)$,
\begin{equation}\label{sXi def}
\aligned
\th_{\pm s}&=\mp\frac1{[s]_q}(\ttg_{1,\pm s}+\cdots+\ttg_{n,\pm s}),\\
\sX_i&=[\ttx_{n-1,0}^-,[\ttx_{n-2,0}^-,\cdots,
[\ttx_{i+1,0}^-,[\ttx_{1,0}^-,\cdots,[\ttx_{i-1,0}^-,
\ttx_{i,1}^-]_{\ttv^{-1}}\cdots]_{\ttv^{-1}}]_{\ttv^{-1}}
]_{\ttv^{-1}}]_{\ttv^{-1}},\;\text{ and}\\
\sY_1&=[\cdots[[\ttx_{1,-1}^+,\ttx_{2,0}^+]_\ttv,\ttx_{3,0}^+]_\ttv,
 \cdots,\ttx_{n-1,0}^+]_\ttv.\endaligned
\end{equation}

The following isomorphism extends Beck's isomorphism.

\begin{Thm}[{\cite[4.4.1]{DDF}}]\label{DDFIsoThm}
 There is a Hopf  algebra isomorphism\footnote{Note that $f|_{\afUslC}$ is the same map as given in \cite[Prop.~6.1]{CP96} as $[x,y]_{q^{-1}}=q^{-1/2}(q^{1/2}xy-q^{-1/2}yx)$. However, the map used in
 \cite[Prop.~2.5]{CP94} is the $f$ here followed by an automorphism of the form $\ttx_{i,s}^\pm\mapsto a^s\ttx_{i,s}^\pm$, $\ttg_{i,t}\mapsto a^t\ttg_{i,t}$, and $\ttk_i\mapsto\ttk_i$ for some $a\in\mbc$.}
 $$f:\DC(n)\lra
\afUglC$$ such that
$$\aligned
&K_i^{\pm1}\lm\ttk_i^{\pm1},\;\quad
E_j\lm \ttx^+_{j,0},\;\quad
F_j\lm \ttx^-_{j,0}
\;(1\leq i\leq n,\,1\leq j<n),\;\\
&E_n\lm \ttv\sX_1\ti\ttk_n,
\quad F_n\lm\ttv^{-1}\ti\ttk_n^{-1}\sY_1,
\quad
\sfz^\pm_s\lm
\mp s\ttv^{\pm s}\th_{\pm
s}\;(s\geq 1).\\
\endaligned$$

\end{Thm}
It is known from \cite[Rem.~6.1]{CP96} that, for every $1\leq i\leq n-1$,
\begin{equation}\label{sXi}
f(E_n)=\ttv\sX_1\ti\ttk_n=(-1)^{i-1}\ttv\sX_i\ti\ttk_n.
\end{equation}

The following elements we define will be used in defining pseudo-highest modules in next section.

For $1\leq j\leq n-1$ and $s\in\mbz$, define the elements $\ms
P_{j,s}\in\afUglC$ through the generating functions
\begin{equation*}
\begin{split}
& \ms P_j^\pm(u):=\exp\bigg(-\sum_{t\geq
1}\frac{1}{[t]_\ttv}\tth_{j,\pm t} (\ttv u)^{\pm
t}\bigg)=\sum_{s\geq 0}\ms P_{j,\pm s} u^{\pm
s}\in\afUglC[[u,u^{-1}]].
\end{split}
\end{equation*}
Note that
\begin{equation}\label{PhiP relation}
\Phi_j^\pm(u)=\ti\ttk_j^{\pm1}\frac{\ms P_j^\pm(\ttv^{-2}u)}{\ms
P_j^\pm(u)}.
\end{equation}

For $1\leq i\leq n$ and $s\in\mbz$, define the elements $\ms
Q_{i,s}\in\afUglC$ through the generating functions
\begin{equation}\label{scrQ}
\begin{split}
&\quad\qquad\ms Q_i^\pm(u):=\exp\bigg(-\sum_{t\geq
1}\frac{1}{[t]_\ttv}\ttg_{i,\pm t} (\ttv u)^{\pm t}\bigg)=\sum_{s\geq
0}\ms Q_{i,\pm s} u^{\pm s}\in\afUglC[[u,u^{-1}]].
\end{split}
\end{equation}
Note that
\begin{equation}\label{relation between P and Q}
\begin{split}
& \ms P_j^\pm(u)=\frac{\ms Q_j^\pm(u\ttv^{j-1})}{\ms
Q_{j+1}^\pm(u\ttv^{j+1})},
\end{split}
\end{equation}
for $1\leq j\leq n-1$.

\section{Simple pseudo-highest weight $\afUglC$-modules}

Let $V$ be a {\it weight} $\afUslC$-module (resp., weight $\afUglC$-module) of type 1. Then $V=\oplus_{\mu\in\mbz^{n-1}}V_\mu$ (resp., $V=\oplus_{\la\in\mbz^{n}}V_\la$) where
$$V_\mu=\{v\in V\mid \ti\ttk_jw=\ttv^{\mu_j}w,\,\forall 1\leq j<n\}\quad(\text{resp.},\,\, V_\la=\{v\in V\mid \ttk_jw=\ttv^{\la_j}w,\,\forall 1\leq j\leq n\}).
$$
Nonzero elements of $V_\la$ are called weight vectors.

Following \cite[12.2.4]{CPbk}, a nonzero weight vector $w\in V$ is called a {\it pseudo-highest weight vector},
\index{pseudo-highest weight vector} if there exist some
$P_{j,s}\in\mbc$ (resp. $Q_{i,s}\in\mbc$) such that
\begin{equation}\label{HWvector}
\ttx_{j,s}^+w=0,\quad\ms P_{j,s} w=P_{j,s} w,\quad(\text{resp.},\,\, \ttx_{j,s}^+w=0,\quad\ms Q_{i,s}w=Q_{i,s}w).
\end{equation}
for all $s\in\mbz$. The module $V$ is called a
{\it pseudo-highest weight module} if $V=\afUslC w$ (resp., $V=\afUglC w$) for some pseudo-highest weight
vector $w$.

Following \cite{FM}, an $n$-tuple of polynomials
$\bfQ=(Q_1(u),\ldots,Q_n(u))$ with constant terms $1$ is called {\it
dominant} if, for $1\leq i\leq n-1$, the ratios
$Q_i(\ttv^{i-1}u)/Q_{i+1}(\ttv^{i+1}u)$ is a polynomial. Let
$\sQ(n)$ be the set of dominant $n$-tuples of polynomials.

For $\bfQ=(Q_1(u),\ldots,Q_{n}(u))\in\sQ(n)$, define
$Q_{i,s}\in\mbc$ ($1\leq i\leq n$, $s\in\mbz$) by the
following formula
$$Q_i^\pm(u)=\sum_{s\geq 0}Q_{i,\pm s}u^{\pm s}.$$
Here
$f^+(u)=\prod_{1\leq i\leq m}(1-a_iu)\iff f^-(u)=\prod_{1\leq i\leq m}(1-a_i^{-1}u^{-1}).$

 Let $I(\bfQ)$
be the left ideal of $\afUglC$ generated by $\ttx_{j,s}^+ , \ms
Q_{i,s}-Q_{i,s},$ and $\ttk_i-\ttv^{\la_i}$, where $1\leq j\leq n-1$,
$1\leq i\leq n$, $s\in\mbz$, and $\la_i=\mathrm{deg}Q_i(u)$. Define Verma type module
$$M(\bfQ)=\afUglC/I(\bfQ).$$
Then $M(\bfQ)$ has a unique (finite dimensional) simple quotient, denoted by $L(\bfQ)$.
The polynomials $Q_i(u)$ are called {\it Drinfeld
polynomials} associated with $L(\bfQ)$.

Similarly, for an $(n-1)$-tuples $\bfP=(P_1(u),\ldots,P_{n-1}(u))\in\sP(n)$ of polynomials with constant terms $1$, define $P_{j,s}\in\mbc$ ($1\leq j\leq n-1$, $s\in\mbz$) as in $P_j^\pm(u)=\sum_{s\geq
0}P_{j,\pm s} u^{\pm s}$ and let $\mu_j=\mathrm{deg}P_j(u)$. Replacing $\ms
Q_{i,s}-Q_{i,s}, \ttk_i-\ttv^{\la_i}$ by $\ms
P_{i,s}-P_{i,s}, \ti\ttk_i-\ttv^{\mu_i}$ in the above construction defines a simple $\afUslC$-module
$\Lcp(\bfP)$. The polynomials $P_i(u)$ are called {\it Drinfeld
polynomials} associated with $\Lcp(\bfP)$.

\begin{Thm}\label{classification of simple afUglC-modules}
$(1)$(\cite{FM}) The $\afUglC$-modules $L(\bfQ)$ with $\bfQ\in\sQ(n)$ are all
nonisomorphic finite dimensional simple polynomial representations
of $\afUglC$. Moreover,
$$L(\bfQ)|_{\afUslC}\cong \bar L(\bfP)$$ where
$\bfP=(P_1(u),\ldots,P_{n-1}(u))$ with
$P_i(u)=Q_i(\ttv^{i-1}u)/Q_{i+1}(\ttv^{i+1}u)$.

$(2)$(\cite{CP91,CPbk}) Let $\sP(n)$ be the set
of $(n-1)$-tuples of polynomials with constant terms $1$. The modules $\Lcp(\bfP)$ with $\bfP\in\sP(n)$ are all nonisomorphic
finite dimensional simple $\afUslC$-modules of  type $1$.
\end{Thm}

Let $\OgC$ (resp., $\OgnC$) be a vector space over $\mathbb C$ with basis $\{\og_i\mid i\in\mathbb Z\}$
(resp., $\{\og_i\mid 1\leq i\leq n\}$). It is a natural $\DC(n)$-module with the action
\begin{equation} \label{QGKMAlg-action}
\aligned
E_i\cdot \og_s&=\dt_{i+1,\bar s}\og_{s-1},\quad F_i\cdot \og_s=\dt_{i,\bar
s}\og_{s+1},\quad
K_i^{\pm 1}\cdot \og_s=\up^{\pm\dt_{i,\bar s}}\og_s,\\
&\sfz_t^+\cdot\og_s=\og_{s-tn},\quad\text{and }\;\;
\sfz_t^-\cdot\og_s=\og_{s+tn}.
\endaligned
\end{equation}
Hence, $\OgnC$ is a $U(n)_\mbc$-module.

 Note that  there is no weight vector in $\OgC$ which is vanished by all $E_i$.
Hence, this is not a highest weight module in the sense of \cite{Lu93}.

The Hopf algebra structure induces a $\DC(n)$-module $\OgC^{\ot r}$,  and hence, an algebra homomorphism
\begin{equation}\label{afzrC}
\afzrC:\DC(n)\lra\text{End}(\OgC^{\otimes r}).
\end{equation}
Similarly, there is an algebra homomorphism
\begin{equation}\label{zrC}
\zrC:U(n)_{\mathbb C}\lra\text{End}(\OgnC^{\otimes r}).
\end{equation}
The images $\afSrC=\text{im}(\afzrC)$ and $\SrC=\text{im}(\zrC)$ are called an affine $q$-Schur algebra and a $q$-Schur algebra, respectively.

The study of representations of affine $q$-Schur algebras \cite[Ch.~4]{DDF} shows that all finite dimensional simple $\afUglC$-modules $L(\bfQ)$
constructed above are simple $\afSrC$-modules. They are nothing but the composition factors of all tensor spaces.

\begin{Thm} Let $\sQ(n)_r=\{\bfQ\in\sQ(n)\mid r=\sum_{i=1}^n\deg Q_i(u)\}$ and let $\text{\rm Irr}(\afSrC)$ denote the set of isoclasses of simple $\afSrC$-modules (and hence, simple $\afUglC$-modules under $\afzrC$). Then $$\bigcup_{r\geq0}\text{\rm Irr}(\afSrC)=\{[L(\bfQ)]\mid \bfQ\in\sQ(n)\},$$
and $\{[L(\bfQ)]\mid \bfQ\in\sQ(n)_r\}$ is a complete set of isoclasses of simple $\afSrC$-modules. Moreover, every $L(\bfQ)$ with $\bfQ\in\sQ(n)_r$ is a quotient (equivalently, subquotient) module of $\OgC^{\ot r}$, and every composition factor of $\OgC^{\ot r}$ is isomorphic to $L(\bfQ)$ for some $\bfQ\in\sQ(n)_r$.
\end{Thm}

\begin{proof} The first assertion is the classification theorem of simple $\afSrC$-modules (see \cite[4.5.8]{DDF}). The second assertion follows from the fact that every finite dimensional $\afSrC$-module is a homomorphic image of $\OgC^{\ot r}$ (\cite[4.6.2]{DDF}). The last assertion is clear since every composition factor of $\OgC^{\ot r}$ is an $\afSrC$-module and every simple $\afSrC$-module is finite dimensional \cite[4.1.6]{DDF}.
\end{proof}

\section{Affine $q$-Schur algebras and their simple representations}

Let $\mathfrak S_{\vtg,r}$ be the {\it affine symmetric group} consisting of all permutations
$w:\mbz\ra\mbz$ such that $w(i+r)=w(i)+r$ for $i\in\mbz$. Then, $\mathfrak S_{\vtg,r}\cong\fS_r\ltimes \mbz^r$, where $\fS_r$ is the symmetric group on $r$ letters. Let $\La(n,r)=\{\la\in\mbn^n\mid r=|\la|\}$ be the set of compositions of $r$ into $n$ parts and, for  $\la\in\La(n,r)$, let $\fS_\la$ be the Young subgroup of $\fS_r$ (or of $\fS_{\vtg,r}$).

Let $\La^+(n,r)$ be the subset of partitions (i.e.~ weakly decreasing compositions) in $\La(n,r)$. Thus, $\La^+(r)=\La^+(r,r)$ is the set of all partitions of $r$. For partition $\la$, let $\la'$ be the {\it dual} partition of $\la$ (so $\la'_i=\#\{j\mid\la_j\geq i\}$).

 Let $\afHrC$ be the Hecke algebra associated with $\mathfrak S_{\vtg,r}$. Thus, $\afHrC$ has a presentation with generators
$T_i,X_j$ $(\text{$i=1,\ldots,r-1$, $j=1,\ldots,r$})$
 and relations
$$\aligned
 & (T_i+1)(T_i-\up^2)=0,\\
 & T_iT_{i+1}T_i=T_{i+1}T_iT_{i+1},\;\;T_iT_j=T_jT_i\;(|i-j|>1),\\
 & X_iX_i^{-1}=1=X_i^{-1}X_i,\;\; X_iX_j=X_jX_i,\\
 & T_iX_iT_i=\up^2 X_{i+1},\;\;X_jT_i=T_iX_j\;(j\not=i,i+1).
\endaligned$$
Let $\HrC$ be the subalgebra generated by all $T_i$. This is the Hecke algebra of $\fS_r$.

Following \cite{VV}, the tensor space $\OgC^{\ot r}$
admits a right $\afHrC$-module structure defined by
\begin{equation}\label{VVaction}
\begin{cases}
\og_{\bf i}\cdot X_t^{-1}
=\og_{i_1}\cdots\og_{i_{t-1}}\og_{i_t+n}\og_{i_{t+1}}\cdots\og_{i_r},\qquad \text{ for all }\bfi\in \mbz^r\\
{\og_{\bf i}\cdot T_k=\left\{\begin{array}{ll} \up^2\og_{\bf
i},\;\;&\text{if $i_k=i_{k+1}$;}\\
q\og_{\bfi s_k},\;\;&\text{if $i_k<i_{k+1}$;}\qquad\text{ for all }\bfi\in I(n,r),\\
q\og_{\bfi s_k}+(\up^2-1)\og_{\bf i},\;\;&\text{if
$i_{k+1}<i_k$,}
\end{array}\right.}
\end{cases}
\end{equation}
where $1\leq k\leq r-1$, $1\le t\le r$, and the action of $\fS_r$ on $I(n,r):=[1,n]^r$ is the place permutation. Apparently, this also defines an action of $\HrC$ on $\OgnC^{\ot r}$.

The formulas of the comultiplication on $\sfz_t^\pm$ \eqref{Hopf} and the first relation in \eqref{VVaction} implies immediately the following.

\begin{Lem}[{\cite[(3.5.5.2)]{DDF}}]\label{action central elts} For any $\bfi=(i_1,\ldots,i_r)\in\mbz^r$,
\begin{equation*}
 \aligned
\sfz_t^+\cdot\og_\bfi&=\sum_{s=1}^r \og_{i_1}\ot\cdots\ot\og_{i_{s-1}}\ot\og_{i_s-tn}\ot\og_{i_{s+1}}\ot\cdots\ot\og_{i_r}=\og_\bfi \sum_{i=1}^rX_i^t,\;\;\text{and}\\
\sfz_t^-\cdot\og_\bfi&=\sum_{s=1}^r \og_{i_1}\ot\cdots\ot
\og_{i_{s-1}}\ot\og_{i_s+tn}\ot\og_{i_{s+1}}\ot\cdots\ot\og_{i_r}=\og_\bfi \sum_{i=1}^rX_i^{-t}.
\endaligned
\end{equation*}
\end{Lem}

We have the following generalization of the fact $\SrC=\text{End}_{\HrC}(\OgnC^{\ot r})$ for $q$-Schur algebras.

\begin{Thm}[{\cite[3.8.1]{DDF}}] \label{onto} The actions of $\DC(n)$ and $\afHrC$ commute and
$$\afSrC=\text{\rm End}_{\afHrC}(\OgC^{\ot r})\; (n\geq2,r\geq1)\;\text{ and }\;\afHrC\cong \text{\rm End}_{\afSrC}(\OgC^{\ot r})^{\rm op}\;(n\geq r).$$
\end{Thm}

Let
$$x_\la:=\sum_{w\in\fS_\la}T_w.$$ Then, there are $\afHrC$-module and $\HrC$-module isomorphisms: $\OgC^{\ot r}\cong \oplus_{\la\in\La(n,r)}x_\la\afHrC$
and $\OgnC^{\ot r}\cong \oplus_{\la\in\La(n,r)}x_\la\HrC$, which induce algebra isomorphisms:
$$\afSrC\cong\text{End}_{\afHrC}(\oplus_{\la\in\La(n,r)}x_\la\afHrC),\quad \SrC\cong\text{End}_{\HrC}(\oplus_{\la\in\La(n,r)}x_\la\HrC).$$

Like the $q$-Schur algebras, Theorem \ref{onto} implies that affine $q$-Schur algebras play a bridging role between representations of quantum affine $\mathfrak {gl}_n$ and affine Hecke algebras. This becomes possible since we have established
the isomorphism $f$ in Theorem \ref{DDFIsoThm} between double Ringel-Hall algebras $\DC(n)$ and the quantum loop algebra $\afUglC$. We now describe how simple $\afSrC$-modules arise from simple $\afHrC$-modules.

A {\it segment} with center $a\in\mbc^*$ and length $k$ is by definition an
ordered sequence
$$[a;k)=(a\ttv^{-k+1},a\ttv^{-k+3},\ldots,a\ttv^{k-1})\in(\mbc^*)^k.$$
A {\it multisegment} is an unordered collections of segments, denoted by formal sum
$$\sum_{i=1}^p[a_i;\nu_i)=[a_1;\nu_1)+\cdots+[a_p;\nu_p),$$
where, possibly, $[a_i;\nu_i)=[a_j;\nu_j)$ for $i\neq j$.

 Let $\mathscr S_r$ be the set of all multisegments of total length $r$:
$$\mathscr S_r=\{[a_1;\nu_1)+\cdots+[a_p;\nu_p)\mid a_i\in\mbc^*, p,\nu_i\geq1, r=\Sigma_i\nu_i\}.$$

For $\bfs=\sum_{i=1}^p[a_i;\nu_i)\in\mathscr S_{r}$, let
$$(s_1,s_2,\ldots,s_r)\in(\mbc^*)^r$$
be the $r$-tuple obtained by juxtaposing the segments in $\bfs$ and let
$J_\bfs$ be the left ideal of $\afHrC$ generated by $X_j-s_j$
for $1\leq j\leq r$.
Then $M_\bfs=\afHrC/J_\bfs$ is a left $\afHrC$-module\footnote{Strictly speaking, the module $M_\bfs$ depends on the order of the segments in $\bfs$. However, the module $V_\bfs$ below does not.} which as an $\sH(r)_\mbc$-module
is isomorphic to the regular representation of $\sH(r)_\mbc$.

Let $\nu=(\nu_1,\ldots,\nu_p)$. After reordering, we may assume that $\nu$ is a partition. Then, the element
$$y_{\nu}=\sum_{w\in\fS_\nu}(-\ttv^2)^{-\ell(w)}T_w\in\afHrC$$
generates the submodule $\afHrC \bar y_\nu$ of $M_\bfs$
which, as an $\sH(r)_\mbc$-module, is isomorphic to $\sH(r)_\mbc
y_\nu$. For each partition $\la$ of $r$, let $E_\la$ be the left cell module defined by the
Kazhdan--Lusztig's C-basis \cite{KL79} associated with the left cell
containing $w_{0,\la}$, the longest element of the Young subgroup $\fS_\la$. Then, as an $\sH(r)_\mbc$-module,
\begin{equation}\label{signed permutation module}
\sH(r)_\mbc y_\nu\cong E_\nu\oplus(\bigoplus_{\mu\vdash r,
\mu\rhd\nu}m_{\mu,\nu}E_\mu).
\end{equation}

Let $V_\bfs$ be the unique composition factor of the $\afHrC$-module
$\afHrC \bar y_\nu$ such that the multiplicity of $E_\nu$ in
$V_\bfs$ as an $\sH(r)_\mbc$-module is nonzero. Note that, if $V_\bfs$ is $\HrC$-irreducible, then
$V_\bfs=E_\nu$. We will use this fact in the last section.

We now can state the following classification theorem due to
Zelevinsky and Rogawski. The
construction above follows \cite{Ro}.

\begin{Thm}\label{classification irr affine Hecke algebra}
Let $\text{\rm Irr}(\afHrC)$ be the set of isoclasses of all simple
$\afHrC$-modules. Then the correspondence $\bfs\mapsto [V_\bfs]$
defines a bijection from $\mathscr S_r$ to $\text{\rm Irr}(\afHrC)$.
\end{Thm}

 Suppose $n>r$, we define a map
\begin{equation}\label{Q_s}
\partial:\ms S_r\lra\sQ(n)_r,\quad \bfs\longmapsto\bfQ_\bfs=(Q_1(u),\ldots,Q_n(u))
\end{equation}
as follows: for $\bfs=\sum_{i=1}^p[a_i;\nu_i)\in\ms S_r$,
let $Q_n(u)=1$ and, for $1\leq i\leq n-1$, define $$Q_i(u)=P_i(u\ttv^{-i+1})P_{i+1}(u\ttv^{-i+2})\cdots
P_{n-1}(u\ttv^{n-2i}),$$ where
$
P_{i}(u)=\prod_{1\leq j\leq p\atop \nu_j=i}(1-a_ju).$

Here $\nu=(\nu_1,\ldots,\nu_p)$ is a partition of $r$. If $\mu_i:=\deg P_i(u)=\#\{j\in[1,p]\mid \nu_j=i\}$ and $\la_i:=\deg Q_i(u)=\#\{j\in[1,p]\mid \nu_j\geq i\}$, then $\la=(\la_1,\ldots,\la_{n-1}, \la_n)=\nu'$, the dual partition of $\nu$, and
$\la_i-\la_{i+1}=\mu_i$ for all $1\leq i<n$.

This map is a bijection which gives the following identification theorem; see \cite[\S\S4.3-5]{DDF}. 

\begin{Thm}\label{n>r representation}
Assume $n>r$. Then we have $\afSrC$-module isomorphisms
$\OgC^{\ot r}\ot_{\afHrC}V_\bfs\cong L(\bfQ_\bfs)$ for all
$\bfs\in\ms S_r$. Furthermore, for any $n$ and $r$,
the set
$$\{\OgC^{\ot r}\ot_{\afHrC}V_\bfs\mid\bfs=[a_1;\nu_1)+\cdots+[a_p;\nu_p)
\in\ms S_r,\,p\geq 1,\,\nu_i\leq n,\,\forall i\}$$ forms a complete set of nonisomorphic simple $\afSrC$-modules.
\end{Thm}

\section{Compatibility of evaluation maps}

Following \cite[Rem.~2]{Ji}, every $a\in\mbc^*$ defines a surjective algebra homomorphism $$\evUa:\afUnC\ra\UnC,\,\,\text{the quantum $\mathfrak {gl}_n$}$$ such that, for all $1\leq i\leq n-1$ and $1\leq j\leq n$,
\begin{equation}\label{evUa}
\begin{split}
(1)&\quad \evUa(E_i)=E_i,\ \evUa(F_i)=F_i,\ \evUa(K_j)=K_j,\\
(2)&\quad \evUa(E_n)=a\ttv^{-1}[F_{n-1}[F_{n-2},\cdots,[F_2,F_1]_{\ttv^{-1}}\cdots]
_{\ttv^{-1}}]_{\ttv^{-1}}K_1K_n,\\
(3)&\quad \evUa(F_n)=a^{-1}\ttv[\cdots[[E_1,E_2]_\ttv, E_3]_\ttv,\cdots,E_{n-1}]_\ttv(K_1K_n)^{-1}.
\end{split}
\end{equation}
This is called the evaluation map at $a$ for quantum affine $\mathfrak {sl}_n$. Note that our definition here is exactly the same as given in \cite[p.316]{CP96} or \cite[Prop.3.4]{CP94} (cf. footnote 1).

For any $a\in\mbc^*$, there is also an evaluation map
$\evHa:\afHrC\to\HrC$ (see, e.g., \cite[5.1]{CP96}) such that
\begin{equation}\label{evHa}
\aligned
\evHa(T_i)&=T_i,\qquad1\leq i\leq r-1,\quad\text{ and}\\
\evHa(X_j)&=a\ttv^{-2(j-1)}T_{j-1}\cdots T_2T_1T_1T_2\cdots T_{j-1}, 1\leq j\leq r.\\
\endaligned
\end{equation}
Following the notation used in \cite[2.1]{JM} with $r=1$ and $T_0=a$ (in the notation there), we will write $L_j:=\evHa(X_j)$.
Note that the elements $a^{-1}L_j=(q-q^{-1})\sL_j+1$ where
$\sL_j=\tilde T_{(1,j)}+\tilde T_{(2,j)}+\cdots+\tilde T_{(j-1,j)}$ ($\tilde T_w=q^{-l(w)}T_w$) are the usual Murphy operators.

We now use the evaluation map $\evHa$ to induce an evaluation map $\evSa$ from the affine $q$-Schur algebra to the $q$-Schur algebra.

First, there is a right $\afHrC$-module isomorphism $\OgC^{\ot r}\cong\OgnC^{\ot r}\ot_{\HrC}\afHrC$.
Second, the evaluation map $\evHa:\afHrC\ra\HrC$ induces a natural $\HrC$-module homomorphism
\begin{equation}\label{ep}
\ep_a:\OgC^{\ot r}\ra\OgnC^{\ot r},\quad xh\longmapsto x\cdot\evHa(h),
\end{equation}
for all $x\in\OgnC^{\ot r}$ and $h\in\afHrC$.

\begin{Prop} The linear map
\begin{equation}\label{evSa}
\evSa:\afSrC\ra\SrC
\end{equation} defined by
$(\evSa(\vi))(x)=\ep_a(\vi(x))$, for any $\vi\in\afSrC$ and $x\in\OgnC^{\ot r}$, is an algebra homomorphism.
\end{Prop}

\begin{proof} By regarding $\HrC$ as a subalgebra of $\afHrC$, one sees easily $\evHa|_{\HrC}$ is the identity map on $\HrC$. This fact implies that, for any $\vi\in\afSrC$, $\evSa(\vi)$ is an $\HrC$-module homomorphism. To quickly see $\evSa(\psi\vi)=\evSa(\psi)\evSa(\vi)$, we may identify $\OgC^{\ot r}$ as the direct sum
$$\fT_\vtg(n,r)=\bigoplus_{\la\in\La(n,r)}x_\la\afHrC$$
of $q$-permutation modules \cite[Lem.~8.3]{VV}, where
$x_\la=T_{\fS_\la}=\sum_{w\in\fS_\la}T_w$, and take two double coset basis elements $\vi=\vi_{\mu\la}^x$ and $\psi=\vi_{\nu\mu}^y$ so that $\vi(x_\la)=T_{\fS_\mu x\fS_\la}=x_\mu T_xh$ and $\psi(x_\mu)=T_{\fS_\nu y\fS_\mu}=x_\nu T_yh'$ for some $h,h'\in\HrC$, we have with $\bar T_z=\evHa(T_z)$
$$
\aligned
\evSa(\psi\vi)(x_\la)&=\ep_a(x_\nu T_yh'T_xh)=x_\nu \bar T_yh'\bar T_xh\\
&=\evSa(\psi)(x_\mu)\bar T_xh
=\evSa(\psi)(x_\mu\bar T_xh)=\evSa(\psi)(\evSa(\vi)(x_\la))\\
&=(\evSa(\psi)\evSa(\vi))(x_\la),\endaligned$$
as desired.\end{proof}

In order to make a comparison of representations, we need a certain compatibility relation between the evaluation maps $\evSa$ and $\evUa$ and the natural homomorphisms $\afzrC$ and $\zrC$ given in \eqref{afzrC} and \eqref{zrC}. In the rest of the section, we establish such a result.

For notational simplicity, we will write the elements in $\OgC^{\ot r}$ by omitting the tensor sign $\otimes$. For $\bfi=(i_1,i_2,\ldots,i_r)\in\mbz^r$, $\la\in\Lanr$, and $1\leq j\leq\la_1$, let
$$\og_\bfi:=\og_{i_1}\og_{i_2}
\cdots\og_{i_r},\quad\ttu_{\la,j}=\og_1^{j-1}\og_n\og_1^{\la_1-j}
\og_2^{\la_2}\cdots\og_n^{\la_n}.$$

\begin{Lem}\label{ttu}
Maintain the notation above. The action of $T_1T_2\cdots T_k$ on $\ttu_{\la,1}$
is given by the formula:
$$\ttu_{\la,1}T_1T_2\cdots T_k=\ttv^k\ttu_{\la,k+1}+(\ttv^2-1)\sum_{1\leq s\leq k}\ttv^{2k-s-1}\ttu_{\la,s},$$
for all $1\leq k\leq\la_1-1$.
\end{Lem}
\begin{proof}
We prove it by induction on $k$. For $k=1$, the result following from the definition \eqref{VVaction}.
Suppose the result holds for any $s$ satisfying $s<k$. Then, by
induction, we have
\begin{equation*}
\begin{split}
\ttu_{\la,1}T_1T_2\cdots T_k&=\ttv^{k-1}\ttu_{\la,k}T_k+(\ttv^2-1)
\sum_{1\leq s\leq k-1}\ttv^{2k-s-3}\ttu_{\la,s}T_k\\
&=\ttv^{k-1}(\ttv \ttu_{\la,k+1}+(\ttv^2-1)\ttu_{\la,k})+
(\ttv^2-1)
\sum_{1\leq s\leq k-1}\ttv^{2k-s-1}\ttu_{\la,s}\\
&=\ttv^{k}\ttu_{\la,k+1}+(\ttv^2-1)\sum_{1\leq s\leq k}\ttv^{2k-s-1}\ttu_{\la,s},
\end{split}
\end{equation*}
proving the formula.
\end{proof}

For $3\leq k\leq n$, let
\begin{equation}\label{ttg}
\ttf_k=[F_{k-1}[F_{k-2},\cdots,[F_2,F_1]_{\ttv^{-1}}
\cdots]_{\ttv^{-1}}]_{\ttv^{-1}}
 \end{equation}
 and let $\ttf_2=F_1$. Then $\evUa(E_n)=aq^{-1}\ttf_nK_1K_n$.
 For $\la\in\Lanr$ and $2\leq k\leq n$, let
 $$\la_{[1,k)}=\la_1+\la_2+\cdots+\la_{k-1}.$$

\begin{Lem}\label{action of ttg}
For $\la\in\Lanr$ and $2\leq k\leq n$. The action of $\ttf_k$ on a tensor of the form
$\og_1^{\la_1}\cdots\og_{k-1}^{\la_{k-1}}\og_{\bfj}$ with $\bfj\in[k,n]^{r-\la_{[1,k)}}$ is given by the formula
$$\ttf_k\cdot\og_1^{\la_1}\cdots\og_{k-1}^{\la_{k-1}}\og_{\bfj}
=\sum_{1\leq s\leq\la_1}\ttv^{1-s}
\og_1^{s-1}\og_k\og_1^{\la_1-s}\og_2^{\la_2}\cdots
\og_{k-1}^{\la_{k-1}}\og_\bfj.$$
\end{Lem}
\begin{proof}Recall that $F_i$ acts on $\OgC^{\ot r}$ via $\Delta^{(r)}(F_i)=\sum_{s=1}^r\underbrace{\tilde K_i^{-1}\ot\cdots\ot\tilde K_i^{-1}}_{s-1}\ot F_i\ot\underbrace{1\ot\cdots\ot1}_{r-s}$. If $k=2$, the action
$F_1\cdot\og_1^{\la_1}\og_{\bfj}=\sum_{1\leq s\leq\la_1}\ttv^{1-s}
\og_1^{s-1}\og_2\og_1^{\la_1-s}\og_\bfj$ ($\bfj\in[2,n]^{r-\la_1}$) follows from
\eqref{QGKMAlg-action}. Assume now $k>2$ and $\bfj\in[k,n]^{r-\la_{[1,k)}}$.
Since $\ttf_k=[F_{k-1},\ttf_{k-1}]_{\ttv^{-1}}$, we have
\begin{equation}\label{action of ttg eq}
\ttf_k \cdot \og_1^{\la_1}\cdots\og_{k-1}^{\la_{k-1}}\og_\bfj=
F_{k-1}\ttf_{k-1}\cdot\og_1^{\la_1}\cdots\og_{k-1}^{\la_{k-1}}\og_\bfj-
\ttv^{-1}\ttf_{k-1}F_{k-1}\cdot\og_1^{\la_1}\cdots\og_{k-1}^{\la_{k-1}}\og_\bfj.
\end{equation}
Since
$$F_{k-1}\cdot\og_1^{\la_1}\cdots\og_{k-1}^{\la_{k-1}}\og_\bfj=\sum_{1\leq t\leq\la_{k-1}}\ttv^{1-t}\og_1^{\la_1}\cdots\og_{k-2}^{\la_{k-2}}
\underbrace{\og_{k-1}^{t-1}\og_{k}\og_{k-1}^{\la_{k-1}-t}\og_\bfj}_{\og_{\bfj'}},$$ where $\bfj'\in[k-1,n]^{r-\la_{[1,k-1)}}$, by induction,
$$\ttf_{k-1}F_{k-1}\cdot\og_1^{\la_1}\cdots\og_{k-1}^{\la_{k-1}}\og_\bfj
=\sum_{1\leq t\leq\la_{k-1}\atop 1\leq s\leq\la_1}\ttv^{2-t-s}
\og_1^{s-1}\og_{k-1}\og_1^{\la_1-s}\og_2^{\la_2}\cdots
\og_{k-2}^{\la_{k-2}}\og_{k-1}^{t-1}\og_k\og_{k-1}^{\la_{k-1}-t}\og_\bfj.$$
On the other hand, by induction again,
\begin{equation*}
\begin{split}
F_{k-1}\ttf_{k-1}\cdot\og_1^{\la_1}\cdots\og_{k-1}^{\la_{k-1}}\og_\bfj
&=\sum_{1\leq s\leq\la_1}F_{k-1}\cdot\ttv^{1-s}\og_1^{s-1}\og_{k-1}\og_1^{\la_1-s}\og_2^{\la_2}\cdots
\og_{k-1}^{\la_{k-1}}\og_\bfj\\
(\text{noting the extra $\og_{k-1}$})&=\sum_{1\leq t\leq\la_{k-1}\atop 1\leq s\leq\la_1}\ttv^{1-s-t}
\og_1^{s-1}\og_{k-1}\og_1^{\la_1-s}\og_2^{\la_2}\cdots
\og_{k-2}^{\la_{k-2}}\og_{k-1}^{t-1}\og_k\og_{k-1}^{\la_{k-1}-t}\og_\bfj
\\
& \qquad+\sum_{1\leq s\leq\la_1}\ttv^{1-s}\og_1^{s-1}\og_{k}\og_1^{\la_1-s}
\og_2^{\la_2}\cdots\og_{k-1}^{\la_{k-1}}\og_\bfj.
\end{split}
\end{equation*}
Substituting into \eqref{action of ttg eq} cancels the double indexed sum and yields the desired formula.
\end{proof}

\begin{Prop}\label{commuting equation}For $a\in\mbc^*$, we have
 $(\zrC\circ\evUa)(E_n)=(\evSa\circ\afzrC)(E_n).$
 \end{Prop}
\begin{proof} We need the check that the images of both sides at $\og_{\bfi_\la}$ are equal for all $\la\in\Lanr$. Since $\Delta^{(r)}(E_i)=\sum_{s=1}^r\underbrace{1\ot\cdots\ot1}_{s-1}\ot E_i\ot\underbrace{\tilde K_i\ot\cdots\ot\tilde K_i}_{r-s}$,
\begin{equation*}
\begin{split}
(\evSa\circ&\afzrC)(E_n)\cdot\og_{\bfi_\la}=\ep_a(E_n\cdot\og_{\bfi_\la})=\ep_a(q^{\la_n-\la_1+j}\sum_{1\leq j\leq\la_1}
\og_1^{j-1}\og_0\og_1^{\la_1-j}\og_2^{\la_2}\cdots\og_n^{\la_n})\\
&=\sum_{1\leq j\leq\la_1}q^{\la_n-\la_1+j}
(\og_1^{j-1}\og_n\og_1^{\la_1-j})\og_2^{\la_2}\cdots\og_n^{\la_n}\cdot
\evHa(X_j)\\
&=a\ttv^{\la_n-\la_1+1}\!\!\!\sum_{1\leq j\leq\la_1}
\og_n\og_1^{\la_1-1}\og_2^{\la_2}\cdots\og_n^{\la_n}T_1T_2\cdots
T_{j-1}=a\ttv^{\la_n-\la_1+1}\sum_{1\leq j\leq\la_1}\ttu_{\la,1}T_1T_2\cdots
T_{j-1}.\\
\end{split}
\end{equation*}
Now applying Lemma~\ref{ttu} yields (noting that the second sum is zero if $j=1$)
\begin{equation*}
\begin{split}
(\evSa\circ\afzrC)(E_n)\cdot\og_{\bfi_\la}&=a\ttv^{\la_n-\la_1+1}\sum_{1\leq j\leq\la_1}
\bigg(\ttv^{j-1}\ttu_{\la,j}+(\ttv^2-1)\sum_{1\leq s\leq j-1}\ttv^{2j-s-3}\ttu_{\la,s}\bigg)\\
&=a\ttv^{\la_n-\la_1+1}
\bigg(\ttv^{\la_1-1}\ttu_{\la,\la_1}+\sum_{1\leq j\leq\la_1-1}\ttv^{j-1}\ttu_{\la,j}+(\ttv^2-1)\sum_{1\leq s\leq \la_1-1\atop s+1\leq j\leq\la_1}\ttv^{2j-s-3}\ttu_{\la,s}\bigg).
\end{split}
\end{equation*}
Since $(\ttv^2-1)\sum_{s+1\leq j\leq\la_1}\ttv^{2j-s-3}=\ttv^{2\la_1-s-1}-\ttv^{s-1}$, it follows that
\begin{equation*}
\begin{split}
(\evSa\circ\afzrC)(E_n)\cdot\og_{\bfi_\la}&=
a\ttv^{\la_n-\la_1+1}\bigg(\ttv^{\la_1-1}\ttu_{\la,\la_1}+\sum_{1\leq s\leq\la_1-1}\ttv^{2\la_1-s-1}\ttu_{\la,s}\bigg)\\
&=a\ttv^{\la_n}\sum_{1\leq s\leq\la_1}\ttv^{\la_1-s}\ttu_{\la,s}.
\end{split}
\end{equation*}
On the other hand, applying Lemma~\ref{action of ttg} yields $$(\zrC\circ\evUa)(E_n)\cdot\og_{\bfi_\la}=a\ttv^{\la_1+\la_n-1}
\ttf_n\cdot\og_{\bfi_\la}=a\ttv^{\la_n}\sum_{1\leq s\leq\la_1}\ttv^{\la_1-s}\ttu_{\la,s}.$$ Hence, $(\evSa\circ\afzrC)(E_n)\cdot\og_{\bfi_\la}=(\zrC\circ\evUa(E_n))\cdot\og_{\bfi_\la}$
for all $\la\in\Lanr$.
\end{proof}

\begin{Rem} We believe that the equation $(\zrC\circ\evUa)(F_n)=(\evSa\circ\afzrC)(F_n)$ holds as well. Thus, the
 following diagram is commutative:
$$\unitlength=1cm
\begin{picture}(5,3.2)
\put(0.2,2.5){$\afUnC$} \put(4,2.5){$U(n)_\mbc$} \put(0.2,0.5){$\afSrC$}
\put(4,0.5){$\SrC$} \put(1.45,2.6){\vector(1,0){2.3}}
\put(1.45,0.75){\vector(1,0){2.3}}
\put(0.9,2.4){\vector(0,-1){1.5}}
\put(4.3,2.4){\vector(0,-1){1.5}} \put(2.5,2.8){$\evUa$}
\put(2.3,0.95){$\evSa$} \put(0.35,1.7){$\afzrC$}
\put(4.4,1.7){$\zrC$}
\end{picture}$$
However, the proof is much more complicated than the $E_n$ case. Fortunately, the Identification Theorem
we will establish requires only the compatibility for the $E_n$ case.
\end{Rem}
\section{A result of Chari--Pressley}

In this section, we reproduce a theorem of Chari--Pressley \cite[3.5]{CP94}.
There are two differences in our approach. First, the isomorphism given in \cite[2.5]{CP94} and used in the proof of \cite[3.5(2)(3)]{CP94} has been replaced by the isomorphism $f$ given in \ref{DDFIsoThm} (see footnote 1). Second, the proof of \cite[3.5]{CP94} used a category equivalence \cite[3.2]{CP94} which turns  a simple ${\rm U}_\mbc(\mathfrak{sl}_n)$-module into a simple ${\rm U}_\mbc(\mathfrak{gl}_n)$-moduleand, hence, a $\afUglC$-module under the evaluation map. We directly start with a simple $q$-Schur algebra module regarded as a simple $U_\mbc(\mathfrak{gl}_n)$-module with a partition as the highest weight.
It turns out that the description of the centers of the segments, which consist of the roots of a Drinfeld polynomial, is considerably simpler under this approach.

For $\la\in\La^+(n,r)$, let $L(\la)$ be the simple $\SrC$-module with highest weight $\la$.

\begin{Thm}\label{compute P(u)}
For $a\in\mbc^*$ and $\la\in\La^+(n,r)$, let $L(\la)_a$ be the inflated $\afUglC$-module by the evaluation map  $\evSa$ in \eqref{evSa}. If $L(\la)_a|_{\afUslC}\cong \bar L(\bfP)$ with $\bfP=(P_1(u),P_2(u),\cdots,P_{n-1}(u))$, then
$$P_j(u)=\prod_{\la_{j+1}+1\leq s\leq\la_j}(1-a\ttv^{2s-1-j}u)$$
for $1\leq j\leq n-1$. In particular, the roots of $P_j(u)$ forms the segment $[a^{-1}q^{j-\la_j-\la_{j+1}};\la_j-\la_{j+1})$ with center $a^{-1}q^{j-\la_j-\la_{j+1}}$ and length $\la_j-\la_{j+1}$. (Note that  $\la_j=\la_{j+1}\implies P_j(u)=1$.)
\end{Thm}
\begin{proof} The proof here follows that of \cite[3.5]{CP94}.  By Lemma \ref{DDFIsoThm}, we identify $\DC(n)$ with $\afUglC$ under the isomorphism $f$. Thus, $F_j=\ttx^-_{j,0}$,
$E_j=\ttx^+_{j,0}$, and
$K_i=\ttk_i$, for all $1\leq j\leq n-1$, $1\leq i\leq n$. With the notation in \eqref{ttg}, we have $\evUa(E_n)=aq^{-1}\ttf_n K_1K_n$.

Fix $i\in[1, n-1]$. Recall from \eqref{sXi def} the notation $\sX_i$. Then \eqref{sXi} becomes $E_n=(-1)^{i-1}q\sX_i\ti\ttk_n$. If $w_0$ is a highest weight vector in $L(\la)_\la$, then, by Proposition~\ref{commuting equation},
$$(-1)^{i-1}\ttv^{\la_n-\la_1+1}\sX_iw_0=
\evSa\circ\afzrC(E_n)\cdot w_0=\zrC\circ\evUa(E_n)\cdot w_0=a\ttv^{\la_1+\la_n-1}\ttf_nw_0.$$
Thus,
\begin{equation}\label{compute P(u) eq1}
\sX_iw_0 =(-1)^{i-1}
a\ttv^{2(\la_1-1)}
\ttf_nw_0.
\end{equation}
Let $\mu=(\mu_1,\cdots,\mu_{n-1})$ with $\mu_i=\la_i-\la_{i+1}$ and define recursively
$$\sX_{i,j}=\begin{cases}[F_1,[F_2,\cdots,[F_{i-1},
\ttx_{i,1}^-]_{\ttv^{-1}}\cdots]_{\ttv^{-1}}]_{\ttv^{-1}},&\text{ if }j=i;\\
[F_j,\sX_{i,j-1}]_{\ttv^{-1}},&\text{ if }i+1\leq j\leq n-1.\\
\end{cases}$$
Then, $\sX_{i,n-1}=\sX_i$ and
$$\aligned
E_{n-1}\sX_i&=E_{n-1}F_{n-1}\sX_{i,n-2}-\ttv^{-1}\sX_{i,n-2}E_{n-1}F_{n-1}\\
&=\bigl(F_{n-1}E_{n-1}+\frac{\ti\ttk_{n-1}-\ti\ttk_{n-1}^{-1}}{\ttv-\ttv^{-1}}\bigr)\sX_{i,n-2}-
\ttv^{-1}\sX_{i,n-2}\bigl(F_{n-1}E_{n-1}+\frac{\ti\ttk_{n-1}-\ti\ttk_{n-1}^{-1}}{\ttv-\ttv^{-1}}\bigr).
\\\endaligned$$
By (QGL2), we see that $\frac{\ti\ttk_{n-1}-\ti\ttk_{n-1}^{-1}}{\ttv-\ttv^{-1}}\sX_{i,n-2}=
\sX_{i,n-2}\frac{\up\ti\ttk_{n-1}-\up^{-1}\ti\ttk_{n-1}^{-1}}{\ttv-\ttv^{-1}}$. Hence,
$$E_{n-1}\sX_iw_0=([\mu_{n-1}+1]-\ttv^{-1}[\mu_{n-1}])\sX_{i,n-2}w_0.$$
Inductively, we obtain
\begin{equation*}
E_{i+1}E_{i+2}\cdots E_{n-1}
\sX_iw_0=\prod_{i+1\leq s\leq n-1}([\mu_s+1]-\ttv^{-1}[\mu_s])
\sX_{i,i}w_0
\end{equation*}
and, similarly,
$$\aligned E_{i-1}E_{i-2}\cdots E_{1}\sX_{i,i}w_0&= E_{i-1}E_{i-2}\cdots E_{1}
[F_1,[F_2,\cdots,[F_{i-1},\ttx_{i,1}^-]_{\ttv^{-1}}\cdots
]_{\ttv^{-1}}]_{\ttv^{-1}}w_0\\
&=\prod_{1\leq s\leq i-1}([\mu_s+1]-
\ttv^{-1}[\mu_s])\ttx_{i,1}^-w_0.
\endaligned$$
Hence,
\begin{equation}\label{compute P(u) eq2}
\begin{split}
E_{i-1}\cdots E_2 E_1E_{i+1}E_{i+2}\cdots E_{n-1}
\sX_iw_0&=\prod_{1\leq s\leq n-1\atop s\not=i}([\mu_s+1]-\ttv^{-1}[\mu_s])\ttx_{i,1}^-w_0.
\end{split}
\end{equation}
On the other hand,
\begin{equation*}
\begin{split}
E_{i-1}\cdots E_2E_1E_{i+1}E_{i+2}\cdots E_{n-1}
\ttf_nw_0&=\prod_{i+1\leq s\leq n-1}
([\mu_s+1]-\ttv^{-1}[\mu_s])E_{i-1}\cdots E_2E_1\ttf_{i+1}  w_0\\
&=\prod_{i+1\leq s\leq n-1}
([\mu_s+1]-\ttv^{-1}[\mu_s])\\
&\qquad\times\prod_{1\leq s\leq i-1}
([\mu_s]-\ttv^{-1}[\mu_s+1])F_iw_0.
\end{split}
\end{equation*}
This together with \eqref{compute P(u) eq1} and the fact that
$$\prod_{1\leq s\leq i-1}\frac{[\mu_s]-\ttv^{-1}[\mu_s+1]}
{[\mu_s+1]-\ttv^{-1}[\mu_s]}=(-1)^{i-1}\ttv^{2(\la_i-\la_1)-i+1}$$
implies that
$$\ttx_{i,1}^-w_0=a\ttv^{2\la_i-i-1}F_iw_0.$$
Applying $E_i$ to the above equation and noting $\phi_{i,1}^-=0$ yields
\begin{equation}\label{compute P(u) eq3}
\phi_{i,1}^+w_0=a\ttv^{2\la_i-i-1}(\ttv^{\mu_i}-
\ttv^{-\mu_i})w_0.
\end{equation}
By the corollary in \cite[3.5]{CP94}, we may assume
$$P_i(u)=\prod_{1\leq j\leq\mu_i}(1-b_i\ttv^{2j-\mu_i-1}u).$$
Thus,
$$\ttv^{\mu_i}\frac{P_i(\ttv^{-2}u)}{P_i(u)}
=\ttv^{\mu_i}+\ttv^{\mu_i}b_i(\ttv^{\mu_i-1}-\ttv^{-\mu_i-1})u
+{\bf O}(u^2).
$$
This together with \eqref{PhiP relation} and \eqref{compute P(u) eq3} implies that
$$a\ttv^{2\la_i-i-1}(\ttv^{\mu_i}-
\ttv^{-\mu_i})=\ttv^{\mu_i}b_i(\ttv^{\mu_i-1}-\ttv^{-\mu_i-1}).$$
Hence, $b_i=a\ttv^{\la_i+\la_{i+1}-i}$ and
$$P_i(u)=\prod_{1\leq j\leq\mu_i}(1-a\ttv^{2(\la_{i+1}+j)-1-i})
=\prod_{\la_{i+1}+1\leq s\leq\la_i}(1-a\ttv^{2s-1-i}u),$$
as required.
\end{proof}

\section{An Identification Theorem}

We now compute the dominant Drinfeld polynomials $\bfQ=(Q_1(u),\ldots,Q_n(u))$ such that
the $\afUglC$-module $L(\la)_a\cong L(\bfQ)$.
By Theorem \ref{compute P(u)}, it remains to compute $Q_n(u)$. This will be done by the action of the central elements $\sfz_t^\pm$ on a highest weight vector $w_0\in L(\la)$.

We first apply a result of James--Mathas to compute the action of  $\sfz_t^\pm$ on the simple $\SrC$-module $L(\la)$ via the evaluation map $\evSa$ in \eqref{evSa}.

The Hecke algebra $\HrC$ of the symmetric groups $\fS_r$ admits a so-called Murphy's basis \cite{Mur}
$$\{x_{\tts,\ttt}:=T_{d(\tts)}^*x_\la T_{d(\ttt)}\mid \tts,\ttt\in\sT^s(\la),\la\in\La^+(r,r)\},$$ where $\sT^s(\la)$ is the set of all standard $\la$-tableaux, $*$ is the anti-involution satisfying $T_i^*=T_i$, and $d(\ttt)$ is the permutation mapping the standard $\la$-tableau $\ttt^\la$ (obtained by filling $1,2,\ldots,r$ from left to right down successive row) to $\ttt$. The subspace $\sH^{\trir\la}$ of $\HrC$ spanned by
$\{x_{\tts\ttt}\mid \tts, \ttt\in \sT^{s}(\mu),
\mu\trir\la\}$ is a two sided ideal of $\HrC$. Let $S^\la=x_\la\HrC/(x_\la\HrC\cap\sH^{\trir\la})$. Then
$S^\la\cong y_{\la'}T_{w_{\la'}} x_\la\HrC$ (see, e.~g., \cite[\S3]{DJ2}).

Similarly, for partition $\la$ of $r$, let $\sT^{ss}(\la,n)$ (resp. $\sT^{ss}(\la,\mu)$) be the set of all semistandard $\la$-tableaux with content in $[1,n]$ (resp., of content $\mu$). Then, the tensor space $\OgnC^{\ot r}$ can be identified with
$$\fT(n,r)=\oplus_{\mu\in\La(n,r)}x_\la\HrC,$$ and the Murphy basis induces
a basis (see, e.g., \cite{JM})
$$\{\fkm_{\ttS,\ttt}\mid (\ttS,\ttt)\in\sT^{ss}(\la,n)\times\sT^s(\la),\,\forall \la\in\La^+(n,r)\}.$$
Fix a linear ordering on
$\La^+(n,r)=\{\la^{(1)},\la^{(2)},\ldots,\la^{(N)}\}$ which refines
the dominance ordering $\trianglerighteq$, i.e.,
 $\la^{(i)}\vartriangleright\la^{(j)}$ implies $i<j$.
For each $1\le i\le N$, let $\fT_i$ denote the
subspace of $\fT$ spanned by all $\fkm_{\ttS,\ttt}$ such that
$(\ttS,\ttt)\in\sT^{ss}(\la,n)\times\sT^s(\la)$ for some $\la\in\{\la^{(1)},\ldots,\la^{(i)}\}$. Then we obtain a
filtration by $\SrC$-$\HrC$-subbimodules:
\begin{equation}
0=\fT_0\subseteq
\fT_1\subseteq\cdots\subseteq\fT_N=\fT(n,r)\end{equation}
such that
$\fT_i/\fT_{i-1}\cong L(\la^{(i)})\otimes S^{\la^{(i)}}$ as $\SrC$-$\HrC$-bimodules.

For $\la=\la^{(i)}$, $\ttS\in \sT^{ss}(\la,n)$, $\tts\in\sT^s(\la)$, $\ttT^\la\in\sT^{ss}(\la,\la)$, and $\ttt^\la$ as above, let
$\vi_{\ttS}=\fkm_{\ttS,\ttt^\la}+\fT_{i-1}$ and $\psi_{\tts}=\fkm_{\ttT^\la,\tts}+\fT_{i-1}$. Then
$$W^\la:=\SrC\vi_{\ttT^\la}$$ is a simple $\SrC$-module, isomorphic to $L(\la)$, and $\{\vi_\ttS\}_{\ttS\in \sT^{ss}(\la,n)}$ forms a basis for $W^\la$.
Since $\fkm_{\ttT^\la,\ttt^\la}=x_\la$, $\vi_{\ttT^\la}\HrC=\psi_{\ttt^\la}\HrC\cong S^\la$ with basis $\{\psi_\tts\}_{\tts\in\sT^s(\la)}$.

Since $\sfz_t^\pm$ are central elements in $\afUglC$, it follows that $(\evSa\circ\afzrC)(\sfz_t^\pm)$ are central in $\SrC$. By Schur's Lemma, $(\evSa\circ\afzrC)(\sfz_t^\pm)$ acts on $W^\la$ by a scalar $c_t^\pm(\la)$. We now compute this scalar.

\begin{Lem} \label{action of central element on Weyl module} Let $\la\in\La^+(n,r)$ and $t$ a positive integer. If $(\evSa\circ\afzrC)(\sfz_t^\pm)$ acts on $W^\la$ by $c_t^\pm(\la)\in\mbc$, then
$$c_t^\pm(\la)=a^{\pm t}\sum_{1\leq i\leq n\atop 1\leq j\leq\la_i}\ttv^{\pm
2t(j-i)}.$$
\end{Lem}
\begin{proof} Under the $\afHrC$-module isomorphism $\OgC^{\ot r}\cong\fT_\vtg(n,r)$ and its restriction giving an $\sH(r)_\mbc$-module isomorphism
$\OgnC^{\ot r}\cong\fT(n,r)$, we identify the two $\afHrC$-modules and the two $\HrC$-modules. In particular, the tensor $\og_{\bfi_\la}$ identifies $\up^{-\ell(w_{0,\la})}x_{\ttT^\la,\ttt^\la}=\up^{-\ell(w_{0,\la})}x_\la$,
 where
$\bfi_\la=(\underbrace{1,\ldots,1}_{\la_1},\ldots,\underbrace{n,\cdots,n}_{\la_n})$ and  $w_{0,\la}$ is the longest element in $\fS_\la$.

Suppose $(\evSa\circ\afzrC)(\sfz_t^\pm)\cdot \vi_{\ttT^\la}=c_t^\pm(\la)\vi_{\ttT^\la}.$ Then
\begin{equation}\label{action of central element on Weyl module eq}
(\evSa\circ\afzrC)(\sfz_t^\pm)(x_\la)\equiv c_t^\pm(\la) x_\la\mod \fT_{i-1}.
\end{equation}
On the other hand, by Lemma \ref{action central elts},
$$(\evSa\circ\afzrC)(\sfz_t^\pm)(x_\la)=\evSa(\afzrC(\sfz_t^\pm))(x_\la)=
\ep_a(\sfz_t^\pm\cdot x_\la)=\ep_a(x_\la\cdot\sum_{s=1}^rX_t^{\pm t})=x_\la\cdot\sum_{s=1}^rL_t^{\pm t},$$
where $L_s=\evHa(X_s)\in\HrC$ for all $1\leq s\leq r$.
By \cite[3.7]{JM}, we have
$$x_\la\cdot L_s\equiv\res_{\ttt^\la}(s)x_\la\mod x_\la\HrC\cap\sH^{\trir\la},$$
where, if $s$ in $\ttt^\la$ is at row $i$ and column $j$, then $\res_{\ttt^\la}(s)=a\ttv^{2(j-i)}$ is the residue\footnote{The $q, Q_1$ in \cite{JM} are $q^2, a$ here.} at $s$. Thus,
$$(\evSa\circ\afzrC)(\sfz_t^\pm)(x_\la)
=\sum_{1\leq s\leq r}x_\la\cdot L_s^{\pm t}\equiv
\sum_{1\leq s\leq r}(\res_{\ttt^\la}(s))^{\pm t}x_\la\mod\fT^{i-1}.$$
Comparing this with \eqref{action of central element on Weyl module eq} yields
$$c_t^\pm(\la)=\sum_{1\leq s\leq r}(\res_{\ttt^\la}(s))^{\pm t}=a^{\pm t}\sum_{1\leq i\leq n,\, 1\leq j\leq\la_i}\up^{\pm
2t(j-i)},$$
as desired.
\end{proof}

\begin{Thm}\label{L(la)a}
For $a\in\mbc^*$ and $\la\in\La^+(n,r)$, suppose $L(\la)_a\cong L(\bfQ)$ for some $\bfQ\in\ms Q(n)_r$. Then $\bfQ=(Q_1(u),\cdots,Q_m(u),1,\ldots,1)$, where $m$ is the number of parts of $\la$, $$\frac{Q_i(u\ttv^{i-1})}{Q_{i+1}(u\ttv^{i+1})}=\prod_{\la_{i+1}+1\leq s\leq\la_i}(1-a\ttv^{2s-1-i}u),$$
for all $1\leq i\leq m-1$, and $Q_m(u)=\prod_{1\leq s\leq\la_m}(1-a\ttv^{2(s-m)}u)$.
\end{Thm}
\begin{proof}
Since $\evSa$ is surjective, $L(\la)_a$ is an irreducible $\afSrC$-module. Thus, by \cite[4.5.8]{DDF}, there exists $\bfQ=(Q_1(u),\cdots,Q_n(u))\in\sQ(n)_r$ such that $L(\la)_a\cong L(\bfQ)$. For $1\leq j\leq n-1$,
let $$P_j(u)=\frac{Q_j(u\ttv^{j-1})}{Q_{j+1}(u\ttv^{j+1})}.$$
By Theorems \ref{classification of simple afUglC-modules} and \ref{compute P(u)} we have, for $1\leq j\leq n-1$,
$$P_j(u)=\prod_{\la_{j+1}+1\leq s\leq\la_j}(1-a\ttv^{2s-1-j}u).$$
Thus, if $m<n$, then $P_j(u)=1$, for all $m+1\leq j< n$. Hence, $Q_i(u)=1$ for all $m<i\leq n$ and
\begin{equation}\label{Q_i's}
\begin{tabular}{r r}
$Q_1(u)=$&$P_1(u)P_2(u\ttv)\cdots P_{m-1}(u\ttv^{m-2})P_m(u\ttv^{m-1}),$\\
$Q_2(u)=$&$P_2(uq^{-1})\cdots P_{m-1}(u\ttv^{m-4})P_m(u\ttv^{m-3}),$\\
$\cdots$&$\cdots$\qquad\qquad\qquad$\cdots$\qquad\\
$Q_{m-1}(u)=$&$P_{m-1}(u\ttv^{-m+2})P_m(u\ttv^{-m+3}),$\\
$Q_{m}(u)=$&$P_m(u\ttv^{-m+1}).$\\
\end{tabular}
\end{equation}
In particular, $Q_m(u)=P_m(u\ttv^{-m+1})=\prod_{1\leq s\leq\la_m}(1-a\ttv^{2(s-m)}u)$, as required in this case.

We now assume $m=n$. Then we have recursively,
$$Q_i (u)=P_i (u\ttv^{-i+1})P_{i+1} (u\ttv^{-i+2})\cdots
P_{n-1} (u\ttv^{n-2i})Q_n (u\ttv^{2(n-i)}),$$
for all $1\leq i\leq n$. We now compute $Q_n(u)$.

Let $w_0$ be a nonzero vector in $L(\la)_\la$. Since $\sfz_t^+=\frac{t\ttv^{ t}}{[t]_\ttv}\sum_{1\leq i\leq n}\ttg_{i, t}$ under the isomorphism $f$ in Lemma \ref{DDFIsoThm}, Lemma \ref{action of central element on Weyl module} implies
$$
\frac{t\ttv^{ t}}{[t]_\ttv}\sum_{1\leq i\leq n}\ttg_{i, t}\cdot w_0=c_t^+(\la) w_0
=(\evSa\circ\afzrC(\sfz_t^+)) w_0=a^t\sum_{1\leq i\leq n\atop
1\leq j\leq\la_i}\ttv^{2t(j-i)}w_0.
$$
 Thus, by \eqref{HWvector}, $\ms Q_i^+(u)\cdot w_0=Q_{i,s}w_0$ for all $s\geq0$ give an identity in $L(\la)[[u]]$:
 $$\prod_{1\leq i\leq n} Q_i(u)w_0=\prod_{1\leq i\leq n}\ms
Q_i^+(u)\cdot w_0.$$
However, by \eqref{scrQ},
\begin{equation*}
\begin{split}
\prod_{1\leq i\leq n}\ms
Q_i^+(u)\cdot w_0&=\exp\bigg(-\sum_{t\geq
1}\frac{1}{[t]_\ttv}\bigg(\sum_{1\leq i\leq n}\ttg_{i, t}\bigg)
(u\ttv)^{ t}\bigg)w_0\\
&=\exp\bigg(-\sum_{1\leq i\leq n\atop 1\leq j\leq\la_i}\sum_{t\geq
1}\frac 1 t(au\ttv^{2(j-i)})^{t}\bigg)w_0\\
&=\prod_{1\leq i\leq n\atop 1\leq j\leq\la_i}\exp\bigg(-\sum_{t\geq
1}\frac 1 t (au\ttv^{2(j-i)})^{t}\bigg)w_0\\
&=\prod_{1\leq i\leq n\atop 1\leq j\leq\la_i}\bigg(1-
au\ttv^{2(j-i)}\bigg)w_0.
\end{split}
\end{equation*}
Hence,
\begin{equation}\label{one hand}
\prod_{1\leq i\leq n} Q_i (u)=\prod_{1\leq i\leq n}\prod_{1\leq j\leq\la_i}\bigg(1-
a\ttv^{2(j-i)}u\bigg).
\end{equation}
Since
$Q_i (u)=P_i (u\ttv^{-i+1})P_{i+1} (u\ttv^{-i+2})\cdots
P_{n-1} (u\ttv^{n-2i})Q_n (u\ttv^{2(n-i)}),$
we have
\begin{equation}\label{another hand}
\begin{split}
\prod_{1\leq i\leq n} Q_i(u) &=\prod_{1\leq k\leq n-1}\bigg( P_k(u\ttv^{-k+1})P_k(u\ttv^{-k+3})\cdots P_k (u\ttv^{k-1})\bigg)\prod_{0\leq l\leq n-1}Q_n (u\ttv^{2l}).
\end{split}
\end{equation}
Now,
\begin{equation*}
\begin{split}
\prod_{1\leq k\leq n-1}
\bigg(P_k (u\ttv^{-k+1})P_k (u\ttv^{-k+3})\cdots P_k (u\ttv^{k-1})
\bigg)
&=\prod_{1\leq k\leq n-1}\prod_{1\leq i\leq k\atop\la_{k+1}+1\leq j\leq\la_k}(1-
a\ttv^{2(j-i)}u)\\
&=\prod_{1\leq i\leq n-1}\prod_{i\leq k\leq n-1\atop\la_{k+1}+1\leq j\leq\la_k}(1-
a\ttv^{2(j-i)}u)\\
&=\prod_{1\leq i\leq n-1}\prod_{\la_n+1\leq j\leq\la_i}(1-
a\ttv^{2(j-i)}u).
\end{split}
\end{equation*}
This together with \eqref{one hand} and \eqref{another hand} implies that
$$\prod_{0\leq l\leq n-1}Q_n (u\ttv^{2l})=
\prod_{1\leq i\leq n}\prod_{1\leq j\leq\la_n}(1-
a\ttv^{2(j-i)}u).$$
Hence, we have in $L(\la)[[u]]$:
$$\prod_{0\leq l\leq n-1}\ms
Q_n(u\ttv^{2l})\cdot w_0=\prod_{0\leq l\leq n-1}
Q_n(u\ttv^{2l}) w_0=\prod_{1\leq i\leq n\atop 1\leq j\leq
\la_n}(1-a\ttv^{2(j-i)}u)w_0.$$
On the other hand, by \eqref{scrQ} again,
\begin{equation}\label{Qn}
\prod_{0\leq l\leq n-1}\ms
Q_n(u\ttv^{2l})\cdot w_0=\exp\bigg(-\sum_{t\geq 1}\frac{1}{[t]_\ttv}\ttg_{n,
t}\bigg(\sum_{0\leq l\leq n-1}\ttv^{2lt}(u\ttv)^{
t}\bigg)\bigg)\cdot w_0
\end{equation}
Applying ln$(\,\,)$ formerly yields,
\begin{equation*}
\begin{split}
-\sum_{t\geq 1}\frac{1}{[t]_\ttv}\ttg_{n,
t}\bigg(\sum_{0\leq l\leq n-1}\ttv^{ (2l+1)t}\bigg)u^{t}\cdot w_0&=\sum_{1\leq i\leq n\atop 1\leq j\leq
\la_n}\ln(1-a\ttv^{2(j-i)}u)w_0\\
&=-\sum_{t\geq 1}\bigg(\sum_{1\leq i\leq n\atop 1\leq j\leq
\la_n}\frac 1 t (a\ttv^{2(j-i)})^t\bigg)u^tw_0.
\end{split}
\end{equation*}
Equating coefficients of $u^t$ gives
$$\frac{1}{[t]_\ttv}\ttg_{n,
t} \sum_{0\leq l\leq n-1}\ttv^{ (2l+1)t}\cdot w_0
=\frac 1 t \sum_{1\leq i\leq n\atop 1\leq j\leq
\la_n}(a\ttv^{2(j-i)})^t w_0.$$
Since
$$\sum_{1\leq j\leq\la_n}\sum_{i=1}^n
(a\ttv^{2(j-i)})^t=\sum_{1\leq j\leq\la_n}\frac{a^t\ttv^{2(j-1)t}(1-\ttv^{-2tn})}{1-\ttv^{-2t}}$$
and $\sum_{0\leq l\leq n-1}\ttv^{(2l+1)t}=\frac{\ttv^{2tn}-1}{\ttv^t-\ttv^{-t}},$
it follows that
$$\ttg_{n,t}\cdot w_0=\frac{[t]_q}{t}\sum_{1\leq j\leq\la_n} (a\ttv^{ 2j-2n-1})^t w_0.$$
Substituting it in \eqref{Qn} gives
\begin{equation*}
\begin{split}
\ms Q_n(u)\cdot w_0&=\exp\bigg(-\sum_{t\geq 1}\frac 1 t
\sum_{1\leq j\leq\la_n}(a\ttv^{2(j-n)}u)^t\bigg)w_0\\
&=\prod_{1\leq j\leq\la_n}(1-a\ttv^{2(j-n)}u)w_0.
\end{split}
\end{equation*}
Hence, we obtain $Q_n(u)=\prod_{1\leq j\leq\la_n}(1-a\ttv^{2(j-n)}u)$.
\end{proof}

Now, \eqref{Q_i's} together with Theorem~\ref{compute P(u)} implies immediately the following.
\begin{Cor}\label{DPforSR}
Suppose $a\in\mbc^*$ and $\la\in\La^+(n,r)$ has $m$ parts. If $L(\la)_a\cong L(\bfQ)$ for some $\bfQ\in\ms Q(n)_r$, then, for all $1\leq i\leq m$, $Q_i(u)$ is the polynomial with degree $\la_i$, constant 1, and roots forming the segment $[a^{-1}\ttv^{-\la_i+2i-1};\la_i)$.
\end{Cor}

\section{Application to affine Hecke algebras}

 By the evaluation map $\evHa:\afHrC\to\HrC$ defined in \eqref{evHa}, every $\HrC$-module $N$ defines a $\afHrC$-module $N_a$. We now identify the simple $\afHrC$-modules $(E_\mu)_a$ for every partition $\mu$ and $a\in\mbc^*$ in terms of multisegments in $\ms S_r$. Recall from \cite[Th.~1.4]{KL79} that the left cell modules $E_\la$ ($\la\in\La^+(r)$) defined in \eqref{signed permutation module} form a complete set of simple $\HrC$-modules.

By Theorem \ref{L(la)a}, we may define a map
$$\ti\partial:\La^+(n,r)\times\mbc^*\lra \sQ(n)_r,\quad (\la,a)\longmapsto \bfQ(\la,a)=(Q_1(u),\ldots,Q_n(u)),$$
where $Q_n(u)=\prod_{1\leq k\leq\la_n}(1-a\ttv^{2(k-n)}u)$ and
$P_i(u)=\frac{Q_i(u\ttv^{i-1})}{Q_{i+1}(u\ttv^{i+1})}=\prod_{\la_{i+1}+1\leq k\leq\la_i}(1-a\ttv^{2k-1-i}u)$ for $1\leq i\leq n-1$.
If $n> r$ and $\la=(\la_1,\ldots,\la_m)$ has $m(\leq r)$ parts, then  $\bfQ(\la,a)$ has the form
$(Q_1(u),\ldots,Q_m(u),1,\ldots,1)$ as in \eqref{Q_i's}.
We now compute $\partial^{-1}(\bfQ(\la,a))$.
\begin{Lem} Let $n>r$. Suppose the map $\partial^{-1}\circ\ti\partial:\La^+(r)\times\mbc^*\lra\ms S_r$
is given by $(\la,a)\longmapsto \bfs(\la,a)$ and $\la$ has $m$ parts. Then
$$\bfs(\la,a)=\sum_{i=1}^m\sum_{\la_{i+1}+1\leq k\leq\la_i}[a\ttv^{2k-1-i};i),$$
and the partition associated with $\bfs(\la,a)$ is $\la'$.
\end{Lem}
\begin{proof}
Since the (inverses of the) roots of $P_i(u)$ are
$$\{a\ttv^{2k-1-i}\mid  \la_{i+1}+1\leq k\leq\la_i\},$$
by the definition of $\partial$, $\bfs(\la,a)$ consists of $\la_i-\la_{i+1}$ segments with length $i$ and centers $a\ttv^{2k-1-i}, \la_{i+1}+1\leq k\leq\la_i$, for all $1\leq i\leq m$.
\end{proof}

Recall from \eqref{signed permutation module} that, for each $\mu\in\La^+(r)$,  $E_\mu$ is the left cell module  for $\HrC$
defined by
Kazhdan--Lusztig's C-basis \cite{KL79} associated with the left cell containing $w_{0,\mu}$, where $w_{0,\mu}$ is the longest element in $\fS_\mu$. If $S_\mu$ denotes the Specht module contained in $\sH(r)_\mbc x_\mu$ (so that $S_\mu=\HrC y_{\mu'}T_{w_{\mu'}} x_\mu $), then $S_{\la'}\cong E_\la$.

\begin{Prop}For $a\in\mbc^*$ and $\la\in\La^+(r)$. Let $\bfs=\bfs(\la,a)$ as above. Then $(S_\la)_a\cong V_\bfs$.
\end{Prop}
\begin{proof}Assume $n>r$. Then there is an idempotent $e\in\SrC$ such that $e\SrC e\cong\HrC$. This algebra isomorphism gives rise to the so-called Schur functor from the category of finite dimensional $\SrC$-modules to the category of finite dimensional $\HrC$-modules by sending $N$ to $eN$.
By Theorems~\ref{n>r representation} and \ref{L(la)a}, we have
$$L(\la)_a\cong L(\bfQ)\cong\OgC^{\ot r}\ot_{\afHrC}V_\bfs,$$ where $\bfQ=\bfQ(\la,a)$.
Restriction gives an $\SrC$-module isomorphism $L(\la)\cong\OgnC^{\ot r}\ot_{\HrC}V_\bfs$.
By applying Schur's functor, we see that $\HrC$-module $V_\bfs$ is irreducible. Hence,
$V_\bfs=E_{\la'}$. Therefore, $(S_{\la})_a\cong V_\bfs$.
\end{proof}

\begin{Cor} {\rm (1)} For any partition $\la\in\La^+(r)$ and $n\geq r$, if $L(\la)$ is the simple $\SrC$-module with highest weight $\la$, then $eL(\la)\cong S_{\la}$.

{\rm(2)} If $\bfQ\in \sQ(n)_r$ and $\la=(\deg Q_1,\ldots,\deg Q_n)$, and $\mu$ is a weight of $L(\bfQ)$, then $\mu\unlhd\la$ under the dominance order $\unlhd$.
\end{Cor}
\begin{proof} Statement (1) follows the proof above. It remains to prove the second statement. By applying another type of Schur functor, we may assume $n>r$. Thus, $L(\bfQ)\cong\OgC^{\ot r}\ot_{\afHrC}V_\bfs$, for some
$\bfs\in\ms S_r$ such that $\bfQ=\bfQ_\bfs$, and $\la'$ is the partition associated with $\bfs$. Since, as a $\HrC$-module, $V_\bfs\cong E_{\la'}\oplus(\oplus_{\nu\rhd\la'}n_{\nu,\la}E_\nu)$ by \eqref{signed permutation module}.
Thus, as a $\text{U}_\mbc(\frak{gl}_n)$-module, the first assertion implies
$$L(\bfQ)\cong L(\la)\oplus (\oplus_{\nu\rhd\la'}n_{\nu,\la}\OgnC^{\ot r}\ot S_{\nu'})\cong L(\la)\oplus (\oplus_{\nu\rhd\la'}n_{\nu,\la}L(\nu')),$$
where the second sum is over $\nu$ with $\OgnC^{\ot r}\ot S_{\nu'}\neq 0$.
Now, since $\nu\rhd\la'$ implies $\nu'\lhd\la$ and the weight spaces of $L(\bfQ)$ as $\afUglC$-module or as
$\text{U}_\mbc(\frak{gl}_n)$ are the same, our assertion follows.
\end{proof}

Part (2) of the result above is \cite[Lem.~4.5.1]{DDF}. The proof here is different.

\end{document}